\documentclass[a4paper,11pt]{amsart}
\usepackage{amsmath}  
\usepackage{amssymb}
\usepackage{amsthm}  
\usepackage{amscd}  
\usepackage{enumerate,enumitem}
\usepackage{xcolor}
\usepackage{mathtools} %new package needed http://ctan.org/pkg/mathtools
\usepackage{tikz-cd}

\theoremstyle{plain}
\newtheorem{thm}{Theorem}[section]
\newtheorem{prop}[thm]{Proposition}
\newtheorem{lemma}[thm]{Lemma}
\newtheorem{cor}[thm]{Corollary}

\theoremstyle{definition}
\newtheorem{defn}{Definition}

\theoremstyle{remark}
\newtheorem{rem}{Remark}

\newcommand{\R}{\mathbf R}
\newcommand{\C}{\mathbf C}
\newcommand{\SU}{\mathrm{SU}}

\newcommand{\G}{\mathrm G}

\title[Moduli of Einstein-Hermitian harmonic mappings]
      {Moduli of Einstein-Hermitian harmonic mappings of the projective line into quadrics}

\author[Macia, Nagatomo]{Oscar Macia, Yasuyuki Nagatomo}

\address[OM]{Department of Mathematics, Faculty of Mathematical Sciences, UNIVERSITY OF VALENCIA,
C.Dr Moliner, 50, Burjassot, 46100, Valencia, SPAIN}
\email{oscar.macia@uv.es}

\address[YN]{Department of Mathematics, MEIJI UNIVERSITY, 
Higashi-Mita, Tama-ku, Kawasaki-shi, Kanagawa 214-8571, JAPAN} 
\email{yasunaga@meiji.ac.jp}


\begin{document}

\maketitle

\begin{abstract}
The present article studies the class of Einstein-Hermitian harmonic maps of constant K\"ahler angle from
the projective line into quadrics. 
We provide a description of their moduli spaces up to image, and gauge--equivalence using the language of
vector bundles and representation theory. It is shown that the dimension of the moduli spaces
is independent of the  Einstein-Hermitian constant,
and rigidity of the associated real standard, and totally real maps is examined. Finally, certain classical results
concerning  embeddings of two-dimensional spheres into spheres are rephrased and derived in our formalism.
\end{abstract}

%\tableofcontents

\section{Introduction}
\indent Let $S,Q\to \mathrm{Gr}_{n}(\mathbf{R}^{n+2})$ 
denote respectively the tautological, and universal quotient bundles over the
complex hyperquadric.
Given a Riemannian manifold $(M,g)$ and a mapping $f:M\to \mathrm{Gr}_n(\mathbf{R}^{n+2}),$ write $V$ for $f^{-1}Q.$
Then, %If $f$ is a mapping into the quadric,
the {\it mean curvature operator of $f$} is the bundle endomorphism
$A\in \Gamma(\mathrm{End} V)$ defined \cite{Na-13} as:
\[A=\mathrm{Tr}(H \circ  K)\]
where $H,K$ are respectively the pull-backs of the second fundamental forms of 
$S,Q\to \mathrm{Gr}_{n}(\mathbf{R}^{n+2})$, 
and the trace of the $V$-valued two-tensor $H\circ K$  is taken with respect to $g.$
The mapping $f$ is said to be {\it Einstein-Hermitian} (EH, for short) if its mean curvature operator
satisfies the {\it strong} Einstein condition,
\begin{equation}\label{einstein}A=\varphi\; \mathrm{id}_V\end{equation}
for some constant $\varphi,$ which we term the EH-constant.
In the holomorphic setting $A$ coincides up to a sign with the mean curvature in the sense of Kobayashi \cite{Kob}, where
the strong Einstein condition was introduced to define the notion of Einstein--Hermitian vector bundle. Moreover,
(\ref{einstein}) characterises the minima of the functional %defined by the mean curvature operator
\[\int_M |A|^2 dv_M\]
which generalises some instances of the Yang--Mills functional (\cite{Na-13} \S 4, \cite{Kob}, p.111)
Although we do not pursue the functional approach of EH maps in this article, many of
its properties are examined in \S 3.

The present work deals with the classification (up to suitable notions of equivalence)
of full, EH harmonic maps 
$\mathbf C\mathrm P^1\to \mathrm{Gr}_n(\mathbf R^{n+2})$ of degree $k$ and EH-constant $l,$
with constant K\"ahler angle. %fulfilling the gauge condition.
If the classifying criterium is gauge--equivalence of maps (resp. image--equivalence, aka congruence),
the resulting moduli space is  denoted  by $\mathcal M_{k,l}$ (resp. $\mathbf M_{k,l}$). These moduli spaces
are fully analysed in \S 4.
To avoid long repetitions we will write `($k$,EH($l$))  mapping'  instead of  `EH mapping of degree $k$ and EH-constant $l$'.

It has been conjectured from remarks in \cite{MaNaTa} that the moduli space
$\mathcal M_{k,l},$ %of gauge--equivalence classes of full EH harmonic maps of degree $k$ and
and %that
the moduli space $\mathcal M_k$
of gauge--equivalence classes of holomorphic isometric embeddings $\mathbf C\mathrm P^1\to\mathrm{Gr}_n(\mathbf R^{n+2})$
of degree $k,$ satisfying the gauge condition
%from the complex projective line %$\mathbf{C}\mathrm{P}^1$ into the complex hyperquadric
are diffeomorphic.
The present work (\S 4) decides the question in the affirmative as a consequence of the\\

{\bf Main Theorem 1} (Theorem \ref{gmod1})
{\it  If $f:\C \mathrm P^1\to \mathrm{Gr}_n(\R^{n+2})$ is a full,
($k$,EH($l$)) harmonic map of constant K\"ahler angle, then
}
  \begin{enumerate}
\item  $n\leq 2(|k|+2l).$
  \end{enumerate}
  \emph{Assuming $n$ to be maximal, %Suppose maximal $n.$ L
    let $\mathcal M_{k,l}$ be the moduli space up to gauge equivalence of maps,
    and denote  
    its closure by the inner product by $\overline{\mathcal M_{k,l}}.$ Then,}
\begin{enumerate}[resume]
\item $\mathcal M_{k,l}$ {\it can be regarded as an open bounded convex body 
in} \[\bigoplus_{r=l+1}^{2r\leqq |k|+2l}S^{2(|k|+2l-2r)} \mathbf C^2.\]
\item {\it The boundary points of $\overline{\mathcal M_{k,l}}$ describe those maps whose images 
  are included in some totally geodesic submanifold}
  \[\mathrm{Gr}_p(\mathbf R^{p+2})\subset \mathrm{Gr}_{2(|k|+2l)}(\mathbf R^{2(|k|+2l+1)}),\quad p<2(|k|+2l)\]
\item {\it The totally geodesic submanifold $\mathrm{Gr}_p(\mathbf R^{p+2})$ 
can be regarded as the common zero set of some sections of $Q\to \mathrm{Gr}_{2(|k|+21)}(\mathbf R^{2(|k|+2l+1)})$,
which belongs to $(\mathbf{R}^{p+2})^\perp\subset \Gamma( Q)$.}\\
\end{enumerate}

The above characterisation of $\mathcal M_{k,l}$ coincides with the description of $\mathcal M_k$
given in \cite{MaNaTa}, Theorem 7.4. The key tool in the proof of Main Theorem 1 is the so called {\it contraction operator} $\mathcal C,$
introduced in the same section.

Other connections between different moduli spaces are also studied in this article.
As further 
properties of the contraction operator are explored in \S 5, we use the {\it modified contraction
  operator} $\mathcal{\tilde C}$ 
to 
make clear the relation between moduli spaces $\mathcal M_{k,l}$ for different EH-constants. %, and show
We prove\\

{\bf Main Theorem 2} (Theorem \ref{flight.theorem})
\emph{
There is a one-to-one correspondence %(the \emph{flight of moduli})
  between $\mathcal M_{k,l}$ and $\mathcal M_{k,l-1}$ which associates the gauge--equivalence class of full ($k$,EH($l$)) harmonic maps
  determined by $T=(\mathrm{id} + D)^{\frac12}$ to the gauge--equivalence class of full ($k$,EH($l-1$)) harmonic maps
  determined by}
  \[\left(\mathrm{id}+\frac{|D|_{op}}{|\tilde{\mathcal C}^2(D)|_{op}}\tilde{\mathcal C}^2(D)\right)^{\frac12}\]
  \emph{where $\tilde{\mathcal C}$ denotes the modified contraction operator (\ref{modified}).}
  \newline
  
The following two sections describe rigidity results for certain classes of mappings naturally associated
to EH harmonic maps. \S 6 deals with the {\it real standard} map and concludes (Corollary 5.2) that
it is \emph{strongly} rigid, ie admits no deformation at the gauge--equivalence moduli space level.
{\it Totally real} (in the sense of Chen \& Ogiue \cite{Chen.Ogiue}) EH maps are discussed in  \S 7. Strong rigidity
of totally real, full, EH minimal immersions is proved (Theorem 7.1). The dependence of this last result with the strong Einstein condition for the mean curvature operator
is clarified.

Finally, \S 8 diverts from the abstract setting and turns to applications. 
We recover a well--known result by Ruh \& Vilms \cite{Ruh-Vil} on harmonicity of the Gauss map (Theorem 8.2) from our analysis of
the moduli spaces in \S 4.  From it, we are able to give new proofs of classical theorems by Calabi \cite{Cal}, and do Carmo \& Wallach \cite{DoC-Wal} regarding isometric minimal immersions of two-dimensional spheres into spheres (Theorem 8.1).
\let\thefootnote\relax\footnote{
{\it Acknowledgement}: The first--named author would like to thank the hospitality of
Meiji University where part of this work was developed.
The work of the first--named author was supported by the Spanish
Agency of Scienctific and Technological Research (DGICT) and FEDER project MTM2013窶�6961窶撤.
The work of the second--named author was supported by JSPS KAKENHI Grant Number 17K05230. %26400074.
}
%%%%%%%%%%%%%%%%%%%%%%%%%%%%%%%%%%%%%%%%%%%%%%%%%%%%%%%%%%%%%%%%%%%%%%%%%%%%%%%%%%%%%%%%%%%%%%
\section{Mean curvature operator and do Carmo--Wallach theory}
%%%%%%%%%%%%%%%%%%%%%%%%%%%%%%%%%%%%%%%%%%%%%%%%%%%%%%%%%%%%%%%%%%%%%%%%%%%%%%%%%%%%%%%%%%%%%%%

Here we introduce some preparatory material needed in the remaining sections. To avoid repetition of what could be easily found elsewhere, the presentation does not intend to be exhaustive but builds on the preliminaries expounded in \cite{MaNaTa} \S 2. Hence, we make free use of concepts defined there, remarkably the notions of \emph{fullness}, \emph{induced map by $(V\to M,W)$}, \emph{standard map}, \emph{gauge equivalence of maps}, and \emph{image equivalence of maps}.

Let $W$ be a real (oriented) or complex $N$-dimensional vector space together with a fixed scalar product (an inner product or a Hermitian inner product).
Then, we have the following homomorphisms of vector bundles over the Grassmannian $\mathrm{Gr}_p(W)$ of (oriented) $p$-planes in $W$ 
\[\begin{tikzcd}
0\arrow{r} &
S
\rar[to-,
to path={
([yshift=0.5ex]\tikztotarget.west) --
([yshift=0.5ex]\tikztostart.east) \tikztonodes}][swap]{i_S}
\rar[to-,
to path={
([yshift=-0.5ex]\tikztostart.east) --
([yshift=-0.5ex]\tikztotarget.west) \tikztonodes}][swap]{\pi_S}
& 
\underline{W}
\rar[to-,
to path={
([yshift=0.5ex]\tikztotarget.west) --
([yshift=0.5ex]\tikztostart.east) \tikztonodes}][swap]{\pi_Q}
\rar[to-,
to path={
([yshift=-0.5ex]\tikztostart.east) --
([yshift=-0.5ex]\tikztotarget.west) \tikztonodes}][swap]{i_Q}
& 
Q \arrow{r} & 0
\end{tikzcd}\]
where $S\to\mathrm{Gr}_p(W)$ is the \emph{tautological bundle}, $Q\to \mathrm{Gr}_p(W)$ the \emph{universal quotient bundle}, and $\underline W\to \mathrm{Gr}_p(W)$ the trivial vector bundle with fibre $W.$ The natural injection $i_S,$ and projection $\pi_Q$ form a exact sequence.
Moreover,  $\pi_Q$ allows one to regard $W$ a subspace of  the space of sections $\Gamma(Q).$ 
The orthogonal projection defined via the scalar product on $W$ induces the bundle homomorphisms $\pi_S$ and $i_Q,$
together with fibre metrics $g_S,\;g_Q.$

Regarding sections $t\in \Gamma(Q)$
as $W$-valued functions $i_Q(t)$
the differential $di_Q(t)$ splits as
\[di_Q(t)=K(t)+\nabla^Q(t)=\pi_Sdi_Q(t) + \pi_Qdi_Q(t)\]
where $\nabla^Q=\pi_Qdi_Q$ 
is the so-called {\it canonical connection}, 
while the $(Q^{\ast}\otimes S)$ valued $1$-form
$K=\pi_Sdi_Q$ 
is {\it the second fundamental form} 
in the sense of Kobayashi \cite{Kob}. 
In a similar way, a connection 
$\nabla^S=\pi_Sd i_S$ and a second fundamental form $H=\pi_Q d i_S$
are defined. % acting on sections of $S\to Gr_p(W).$
Note that from the identification of the (holomorphic) tangent bundle $T \to \mathrm{Gr}_p(W)$ 
 with $S^{\ast}\otimes Q\to\mathrm{Gr}_p(W),$ the Levi-Civita connection $\nabla$ is induced from 
$\nabla^S$ and $\nabla^Q$. The next Lemma then follows (cf \cite{Na-13}, \S 2):
\begin{lemma}\label{sksymHK} 
If $H,K$ are the second fundamental forms defined above,
  \begin{enumerate}
  \item %The second fundamental forms
    $\nabla H=\nabla K=0$
  \item %The second fundamental forms
$g_Q(H s, t)=-g_S(s,K t).
$
\end{enumerate}
\end{lemma}
Further properties of submanifolds with parallel second fundamental forms have been
studied using the present formalism in \cite{Kog-Nag}.

If $f:M \to \mathrm{Gr}_p(W)$ is smooth, 
we pull back $g_Q$ and $\nabla^Q$ to obtain a fibre metric $g_V$ 
and a connection $\nabla^{V}$ 
on the pull-back of the universal quotient bundle, denoted by $V\to M.$ 
The second fundamental forms $H,K$ are also pulled back 
and denoted by the same symbols.
If we restrict bundle-valued linear forms $H$ and $K$ on 
the pull-back bundle  $f^{-1}T^{\ast} \to M$ 
to linear forms on $M$,
then they are just 
the second fundamental forms  of subbundles 
$f^{-1}S \hookrightarrow \underline{W}$ and $V \hookrightarrow \underline{W}$,
where now $\underline{W}=M\times W$ with certain abuse of the notation.  
The bundle epimorphism $\pi_{V}:\underline{W}\to V$
defines a not necessarily injective linear map 
$W \to \Gamma(V).$
Even if this is the case, we shall still refer to $W\subset\Gamma(V)$
as a {\it space} of sections.

Next, assume that $(M,g)$ is an $m$-dimensional Riemannian manifold 
and let $\{e_i\}_{i=1,\cdots m}$ be an orthonormal frame %field
of $M$.
The Riemannian structure on $M,$ 
and the pull-back connection on $V\to M$ 
define 
the Laplace operator $\Delta^V:\Gamma(V)\to\Gamma(V)$
\[ \Delta^V%=\Delta
=-\sum_{i=1}^n \nabla^{V}_{e_i}\left(\nabla^V\right)(e_i)\] 
We will also introduce 
a bundle homomorphism 
$A \in \Gamma\left(\text{End}\,V \right)$ 
defined as the trace of the composition of the 
second fundamental forms:
$$
A:=\sum_{i=1}^m H_{e_i}K_{e_i}.
$$
We call $A \in \Gamma\left(\text{End}\,V \right)$ 
the {\it mean curvature operator} 
of $f:M\to \mathrm{Gr}_p(W)$. The following properties of $A$ are easily proved (cf \cite{Na-13}, \S 3):

\begin{lemma} \label{Nonposed}Let $A$ be the mean curvature operator of $f:M\to \mathrm{Gr}_p(W)$ as defined above. Then, 
\begin{enumerate}
\item \label{Nonpos} 
  $A$ is a non-positive symmetric (or Hermitian) operator. 
\item \label{ed} 
  The energy density $e(f)$ is equal to $-\mathrm{Tr}\;A$.
\end{enumerate}
\end{lemma}

We use the mean curvature operator $A$ to introduce the concept of %define 
{\it Einstein-Hermitian mapping}: 

\begin{defn}
Let $f$ be a map from a Riemannian manifold into a Grassmannian. Then, 
$f$ is {\it Einstein-Hermitian} %map
(for short, EH map) if 
the mean curvature operator $A$ is proportional to the identity,
ie if \[A=-\mu \;\mathrm{id}_V\] for some non-negative constant $\mu$. 
\end{defn}

To end this section we state the harmonic version of the generalisation of the theorem of
do Carmo--Wallach for reader's benefit.

%%%%%%%%%%%%%%%%%%%%%%%%%%%%%%%%%%%%%%%%%%%%%%%%%%%%%%%%%
\begin{thm}\label{GenDW}\cite{Na-13}
  Let $M=G/K_0$ be a compact reductive Riemannian homogeneous space with decomposition 
$\mathfrak{g}=\mathfrak{k}\oplus \mathfrak{m}$. 
Fix a homogeneous complex line bundle 
$L=G\times_{K_0} V_0 \to M$ 
with 
invariant metric $h$ and 
canonical connection $\nabla$. 
Regard $L \to M$ as a real vector bundle with complex structure $J$.  
If
$f:M \to \mathrm{Gr}_n(\mathbf R^{n+2})$ is a full harmonic map 
satisfying the following two conditions:

\begin{enumerate}
\item[{$\mathrm{(G)}$}]
The pull-back 
$V \to M$ of the universal quotient bundle, 
with the pull-back metric, connection and complex structure  
is gauge equivalent to $L\to M$ with $h$, $\nabla$ and $J$.  

\item[{$\mathrm{(EH)}$}]
The mean curvature operator of $f$ equals $-\mu\;\mathrm{id}_V$
for some positive real number $\mu$, 
such that $e(f)=2\mu $. 
\end{enumerate}

Then there exist an eigenspace $W \subset \Gamma(V)$ of the Laplacian with eigenvalue 
$\mu,$  equipped with $L^2$-scalar product $(\cdot,\cdot)_W,$ 
and a semi-positive Hermitian endomorphism 
$T\in \text{\rm End}\,(W),$ such that 

\begin{enumerate}
\item[$\mathrm{(I)}$]
  The vector bundle $L\to M$ is globally generated by a subspace $\mathbf R^{n+2}$ of $W.$ Denote the inclusion
  by $\iota:\mathbf{R}^{n+2}\to W.$
\item[$\mathrm{(II)}$]
  As a subspace, $\mathbf R^{n+2}=%\text{\rm Ker}\,T^{\bot}$,
  (\ker T)^\bot$
  and the restriction % of
$T|\mathbf{R}^{n+2}$ is a positive Hermitian transformation. %of $\mathbf R^{n+2}$. 

\item[$\mathrm{(III)}$]
Regard $W$ as $\mathfrak{g}$-representation $(\varrho, W).$  
The endomorphism $T$ satisfies 
%%%%%%%%%%%%%%%%%%%%%%%%%%% DW 2 %%%%%%%%%%%%%%%%%%%%%%%%%%%%%%%%%
\begin{eqnarray*}\label{DW 2}
\left(T^2-\mathrm{id}_W, \mathrm{GH}(V_0, V_0)\right)_H&=&0 \\
\label{DW 2b}\left(T^2, \mathrm{GH}(\varrho(\mathfrak m)V_0, V_0)\right)_H&=&0
\end{eqnarray*}
where $V_0$ is regarded as a subspace of $W$ by Frobenius reciprocity and the presence of the scalar product
$(\cdot, \cdot)_W$. 

\item[$\mathrm{(IV)}$]
  The endomorphism $T$ determines an embedding
  \[\mathrm{Gr}_n(\mathbf R^{n+2})\hookrightarrow \mathrm{Gr}_{n^{\prime}}(W),\qquad n^{\prime}=n+%\text{\rm dim}\,\text{\rm Ker}\,
\dim\ker T \]
and 
a fixed bundle isomorphism 
$\phi:L\to V.$
\end{enumerate}
%%%%%%%%%%%%%%%%%%%%%%%%%%%%%%%%%%%%%%%%%%%%%%%%%%%%%%%%%%%%%%%%%%%% 
Then, $f:M \to \mathrm{Gr}_{n}(\mathbf R^{n+2})$ 
can be expressed as 
%%%%%%%%%%%%%%%%%%%%%%%%%%%% DW 3 %%%%%%%%%%%%%%%%%%%%%%%%%%%%%%%%
\begin{equation}\label{DW3} 
f\left([g]\right)=\left(\iota^{\ast}T\iota \right)^{-1}\left(f_0\left([g]\right)\cap (\ker T)^{\bot}\right), 
\end{equation}
where $\iota^{\ast}$ denotes the adjoint operator of $\iota$ under the scalar product on 
$\mathbf R^{n+2}$ induced from $(\cdot,\cdot)_W$ on $W$ and $f_0$ is the standard map by $W$. 
%%%%%%%%%%%%%%%%%%%%%%%%%%%%%%%%%%%%%%%%%%%%%%%%%%%%%%%%%%%%%%%%
The pairs $(f_1, \phi_{1})$ and $(f_2, \phi_{2})$ are gauge equivalent if and only if 
$$
\iota_1^{\ast}T_1\iota_1=\iota_2^{\ast}T_2\iota_2, 
$$
where $(T_i,\iota_i)$ correspond to $f_i$ $(i=1,2)$ under the expression \eqref{DW3}, respectively. 
\end{thm}
The converse of the theorem is also true.
\begin{rem} 
  Let $\tau:\mathrm{Gr}_n(\mathbf R^{n+2}) \to \mathrm{Gr}_n(\mathbf R^{n+2})$ be the map obtained
  by switching the 
orientation of $n$-dimensional subspaces of $\mathbf R^{n+2}$. 
Then $\tau:\mathrm{Gr}_n(\mathbf R_+^{n+2}) \to \mathrm{Gr}_n(\mathbf R_-^{n+2})$ is an  
isometry. 
In the sequel, we do not distinguish a map $f:M \to \mathrm{Gr}_n(\mathbf R^{n+2})$ from 
a map $\tau\circ f:M \to \mathrm{Gr}_n(\mathbf R^{n+2})$. 
\end{rem}

%%%%%%%%%%%%%%%%%%% New section %%%%%%%%%%%%%%%%%%%%%%%%%%%%%%%%%%%%
\section{EH harmonic maps}
%%%%%%%%%%%%%%%%%%%%%%%%%%%%%%%%%%%%%%%%%%%%%%%%%%%%%%%%%%%%%%%%%%%

The universal quotient bundle $Q\to \mathrm{Gr}_n(\mathbf R^{n+2})$ has a holomorphic 
vector bundle structure induced by the canonical connection. 
To obtain a characterisation of  EH harmonic maps
into $\mathrm{Gr}_n(\mathbf R^{n+2})$
we use the natural embedding $i$ of $\mathrm{Gr}_n(\mathbf R^{n+2})$
into  
complex projective space $\mathbf C\mathrm P^n\cong\mathrm{Gr}_{n+1}(\mathbf C^{n+2}).$
Then, the pull-back of the hyperplane bundle $\mathcal O(1)\to \mathrm{Gr}_{n+1}(\mathbf C^{n+2})$ 
is just $Q\to \mathrm{Gr}_n(\mathbf R^{n+2}),$
and the pull-back connection 
is the canonical connection 
by equivariance of  $i.$
Moreover the embedding $i$ induces a real structure on $\mathbf C^{n+2}$, 
which is to be regarded as a space of holomorphic sections of $Q \to \mathrm{Gr}_{n+2}(\mathbf R^{n+2})$. 
The real structure distinguishes a real subspace $\mathbf R^{n+2}\subset \mathbf C^{n+2}$. 
We denote by $J$ the complex structure of $Q \to \mathrm{Gr}_{n+2}(\mathbf R^{n+2})$. 
Thus, if $t\in \mathbf R^{n+2}$, then $Jt \in \mathbf C^{n+2}$. 

\begin{prop}\label{compos}
Let $f:M\to \mathrm{Gr}_n(\mathbf{R}^{n+2})$ be a harmonic map. %of a Riemannian manifold $M$ into $Gr_n(\mathbf R^{n+2})$. 
Then 
$f$ is 
EH harmonic 
if and only if 
the composition $i\circ f:M \to \mathrm{Gr}_{n+1}(\mathbf C^{n+2})$ is a harmonic map %with
of constant energy density. 
\end{prop}

\begin{proof}
Suppose that $f:M \to \mathrm{Gr}_n(\mathbf R^{n+2})$ is an EH harmonic map 
with $A=-\mu\; \mathrm{id}_V$. 
From the generalisation of the Theorem of Takahashi (\cite{Na-13}, \S 3), we see that 
\begin{equation}\label{ttaka} \Delta^V t=\mu t,\end{equation}
for any $t \in \mathbf R^{n+2}.$
Thus, using the complex structure $J$ of $Q\to \mathrm{Gr}_n(\mathbf R^{n+2})$, 
we have that the same is true for any $t\in\mathbf{C}^{n+2}.$
When $Q\to \mathrm{Gr}_n(\mathbf R^{n+2})$ is regarded as a complex vector bundle 
with the canonical connection, 
the pull-back bundle $V \to M$ is also regarded as the pull-back of 
$\mathcal O(1)\to \mathrm{Gr}_{n+1}(\mathbf C^{n+2})$ with the canonical connection 
by the composition of $i$ and $f$. 
Using again the generalisation of the Theorem of Takahashi, the induced map $i\circ f$ 
by the pair $(V\to M,\; \mathbf C^{n+2})$ is harmonic. 
Lemma \ref{Nonposed} yields that $i\circ f$ has  constant energy density. 

Conversely, suppose that 
$i\circ f:M \to \mathrm{Gr}_{n+1}(\mathbf C^{n+2})$ is a harmonic map with constant energy density. 
Since $\mathcal O(1)\to \mathrm{Gr}_{n+1}(\mathbf C^{n+2})$ is of complex rank $1$ 
and $i\circ f$ has constant energy density, 
$i\circ f$ is automatically an EH map with $A=-\mu\;\mathrm{id}_V$. 
It then follows from the generalisation of the Theorem of Takahashi that relation (\ref{ttaka}) 
holds for all $t\in\mathbf{C}^{n+2},$ and also for any $t\in\mathbf{R}^{n+2}\subset\mathbf{C}^{n+2}.$
Hence, if we regard $(i\circ f)^{-1}\mathcal O(1) \to M$ as a 
real vector bundle, 
then we recover $f:M \to \mathrm{Gr}_n(\mathbf R^{n+2})$, 
which is an EH harmonic map by 
the generalisation of the Theorem of Takahashi. 
\end{proof}

In what follows we particularise the previous theory to the case in which $M$ is the complex projective line $\mathbf C\mathrm P^1.$

Let $\mathcal O(k) \to \mathbf C\mathrm P^1$ be the 
holomorphic line bundle of degree $k$ over $\mathbf C\mathrm P^1$ equipped with the 
standard metric and 
canonical connection. 
Using the theory of spherical harmonics \cite{takeuchi}, 
we have a decomposition of $\Gamma(\mathcal O(k))$ in the $L^2$-sense:
\begin{equation}\label{declb1}
\Gamma(\mathcal O(k))=\sum_{l=0}^{\infty} S^{|k|+2l}\mathbf C^2.
\end{equation}
Moreover, $S^{|k|+2l}\mathbf C^2$ is an eigenspace of the Laplacian induced by the canonical connection, and its 
eigenvalue is $2\pi\{2l(|k|+l+1)+|k|\}.$

Let $f:\mathbf C\mathrm P^1\to\mathrm{Gr}_n(\mathbf R^{n+2})$ be a smooth map. %from $\mathbf CP^1$ into $Gr_n(\mathbf R^{n+2})$. 
Then, the pull-back 
of the universal quotient bundle is 
isomorphic, as a complex line bundle, to $\mathcal O(k)\to \mathbf C\mathrm P^1$ 
for some $k\in \mathbf Z,$ 
termed the {\it degree of $f.$} 
In addition,
we say 
that $f$ satisfies the {\it gauge condition} 
(with regard to $\mathcal O(k)\to \mathbf CP^1$ together with its canonical connection)
if there exists a bundle isomorphism preserving %the
metrics and 
connections between 
$V \to \mathbf C\mathrm P^1$ and $\mathcal O(k)\to \mathbf C\mathrm P^1.$ 
This is just condition $\mathrm{(G)}$ in Theorem \ref{GenDW}.

Since the 
bundles are of complex rank one, 
a map 
of degree $k$ satisfies the gauge condition 
if and only if 
\begin{equation}\label{gauge.condition}f^{\ast}\omega_Q=k\omega_1\end{equation}
where $\omega_1$ is the fundamental two-form of $\mathbf C\mathrm P^1,$ and
$\omega_Q$ is the fundamental two-form of $\mathrm{Gr}_n(\mathbf R^{n+2}).$
\begin{rem}\label{rmrk}
If $f:\mathbf C\mathrm{P}^1 \to \mathrm{Gr}_n(\mathbf R^{n+2})$ 
is harmonic, then 
it is conformal
(cf Eells \& Lemaire \cite{E-L}). 
Having constant energy density, 
$f$ is an isometric immersion up to homothety. 
Then, it follows from harmonicity that $f$ is minimal. 
\end{rem}

The gauge condition for maps (\ref{gauge.condition}) is intimately related to the more familiar concept of 
K\"ahler angle (eg Chern \& Wolfson \cite{Chern.Wolfson}, Bolton, Jensen, Rigoli \& Woodward \cite{Bolton.etal}). In the following paragraphs we would like to clarify this relationship.

In the general situation, let $N$ be a K\"ahler manifold with metric $g_n$ and K\"ahler two-form $\omega_n,$ and let $ \phi:\mathbf C\mathrm P^1 \to N$ 
be a harmonic map with constant energy density $e(\phi).$ Regarded as an isometric immersion (up to homothety), it satisfies
$\phi^{\ast}g_n=mg_1$
for some positive number $m>0,$ where $g_1$ denotes the metric on $\mathbf C\mathrm P^1.$
Declaring $\{e_1, e_2\}=\{e,J_1e\}$ 
to be an oriented orthonormal (local) frame of $\mathbf C\mathrm P^1$ 
with respect to $g_1$,  
the homothety condition can be written as
\[
g_n\left(d\phi(e_i),d\phi(e_j)\right)=
m\delta_{ij},\qquad m>0.
\]
In the context of submanifold theory, the {\it K\"ahler angle} $\theta_\phi$ for the map $\phi,$ is defined by the relation
\[ \cos\,\theta_\phi=\frac{g_n\left(J_nd\phi(e), d\phi(J_1e) \right)}{|J_nd\phi(e)|_n|d\phi(J_1e)|_n}.\]
Now, further assume that $\phi$ is subject to a condition similar to (\ref{gauge.condition}). That is,
 to $\phi^*\omega_n=k\omega_1$ for some constant degree $k.$ Then, constancy of the K\"ahler angle follows
 \[ \cos\,\theta_\phi= \frac{k}{m}\]
and $\theta_\phi$ depends only on the degree $k$ and the EH constant $l.$

In the particular case of our interest, 
if $f:\mathbf C\mathrm P^1 \to \mathrm{Gr}_n(\mathbf R^{n+2})$ is 
an EH harmonic map, then 
by
Remark \ref{rmrk}, we can regard it
as an isometric immersion up to homothety.
The previous paragraph shows that for an EH harmonic mapping, satisfying the gauge condition (\ref{gauge.condition})
is equivalent to the more familiar requirement of constancy of the K\"ahler angle. Hence, statements in terms of this latter condition
are thus favoured in the sequel.

%%%%%%%%%%%%%%%%%%%%%%%%%%%%%%%%%%%%%%%%%%%%%%%%%%%%%%
\begin{lemma}
  Let $f:\mathbf C\mathrm P^1\to \mathrm{Gr}_n(\mathbf R^{n+2})$ be an EH harmonic map of degree $k,$ %with
  and of constant K\"ahler angle.   Then there exists a non-positive integer $l$ such that 
$e(f)=4\pi\{2l(|k|+l+1)+|k|\}$.  
\end{lemma}
\begin{proof}
Since $f$ is a harmonic map of degree $k$, 
it follows from the generalisation of the Theorem of Takahashi that 
$\Delta^V t+A t=0$ for 
any section $t$ of $\mathbf R^{n+2} \subset \Gamma(V)$. 
By the EH
condition, we have a non-negative $\mu$ such that 
$A=-\mu\; \mathrm{id}_V$. 
Combined with the gauge condition (\ref{gauge.condition}) implied by the constancy of the K\"ahler
angle assumption, $\mathbf R^{n+2} \subset \Gamma(V)$ 
can be considered as 
the eigenspace with eigenvalue $\mu$ of the Laplacian 
induced by the 
canonical connection, acting on $\Gamma(\mathcal O(k)).$ 
Thus, $\mu=2\pi\{2l(|k|+l+1)+|k|\}$, and Lemma \ref{Nonposed} (\ref{ed}) yields the result. 
\end{proof}

\begin{defn}
  Let $f: \mathbf C\mathrm P^1 \to \mathrm{Gr}_n(\mathbf R^{n+2})$ be an EH harmonic map of degree $k,$ and constant
  K\"ahler angle. 
If the energy density of $f$ is $4\pi\{2l(|k|+l+1)+|k|\},$
then $l$ is said to be the {\it Einstein-Hermitian constant of $f$} (EH constant, for short).
\end{defn}

\begin{rem}
In the above situation, $A=-2\pi\{2l(|k|+l+1)+|k|\}\mathrm{id}_{\mathcal O(k)}$. 
Hence, $2\pi\{2l(|k|+l+1)+|k|\}$ should more accurately be called 
the EH constant of $f,$ in accordance with what stated in \S 1. However, in what follows we adopt the above convention
by simplicity. 
Moreover, we will often be making reference to `Einstein-Hermitian harmonic
mappings of degree $k$ and Einstein-Hermitian constant $l$'. We will shorten this
to `($k,$EH($l$)) harmonic map', instead.
\end{rem}

%%%%%%%%%%%%%%%%%%%%%%%%%%%%%%%%%%%%%%%%%%%%%%%%%%%%%%%%%%%%%%%%%%%%%%%%%%%%%%%%%%%%%%%%%%%%%%%%%%%%%%%%%
\section{Moduli space by gauge, and image equivalence}%%%%%%%%%%%%%%%%%%%%%%%%%%%%%%%%%%%%%%%%%%%%%%%%%%%
%%%%%%%%%%%%%%%%%%%%%%%%%%%%%%%%%%%%%%%%%%%%%%%%%%%%%%%%%%%%%%%%%%%%%%%%%%%%%%%%%%%%%%%%%%%%%%%%%%%%%%%%%

In this section we describe 
the moduli space of gauge equivalence classes of ($k$,EH($l$)) harmonic maps
$\mathbf C\mathrm P^1\to \mathrm{Gr}_n(\mathbf R^{n+2})$ of
constant K\"ahler angle.
  It is a direct application of the generalised version of the theory of do Carmo--Wallach \cite{Na-13}, \S 5. The
  particular version of the theorem needed has been introduced in \S 2 as Theorem \ref{GenDW}.

The proof of our Main Theorem 1 is  similar to the one of the Main Theorem 
in \cite{MaNaTa} except by the use made here of the {\it contraction operator}. 
Hence, we use the same notation and conventions used there without further explanation.
The representation theory needed can be consulted in \cite{MaNaTa} \S 4.

Let $W$ denote 
the space of the $l$-th
eigensections of $\mathcal O(k) \to \mathbf C\mathrm P^1,$ in the ordering defined by (\ref{declb1}), 
regarded as the $\text{SU}(2)$--representation $S^{|k|+2l}\mathbf C^2.$ 
It decomposes under the action of the subgroup $\mathrm U(1)$
as
\[W=\C_{-|k|-2l}\oplus\C_{-|k|-2l+2}\oplus\dots 
\oplus \C_{|k|+2l-2}\oplus\C_{|k|+2l}, \]
where $\mathbf C_{\lambda}$ denotes the irreducible $\mathrm{U}(1)$-module of weight $\lambda$. 

The homogeneous description of $\mathcal{O}(k)\to\mathbf C\mathrm P^1$ is
$\mathrm{SU}(2)\times_{\mathrm U(1)}V_0\to\mathbf C\mathrm P^1$ with $V_0=\mathbf C_{-k}\subset W.$
Following the generalisation of do Carmo--Wallach theory \cite{Na-13} \S 5, we shall regard the universal quotient
bundle as a real vector bundle of rank $2,$ and determine the subspaces 
$\G\mathrm{S}(V_0, V_0)$ and $\G\mathrm{S}(\mathfrak{m}V_0,V_0)$ of $\mathrm{S}(W).$
In the sequel
$V_0$ and $W$  shall stand either for the complex %modules
or 
underlying real vector spaces whenever
the meaning is clear. Since $\G\mathrm{H}(V_0,V_0)\subsetneq
\G\mathrm{S}(V_0,V_0)$, we have that 
$\mathrm{H}(W) \subset \G\mathrm{S}(V_0,V_0)$
and we must determine the intersections 
\[\G\mathrm{S}(V_0,V_0)\cap \sigma\mathrm{H}_+(W)\oplus J\sigma\mathrm{H}_+(W),\] 
and \[\G\mathrm{S}(\mathfrak{m}V_0,V_0)\cap \sigma\mathrm{H}_+(W)\oplus J\sigma\mathrm{H}_+(W).\]

\begin{lemma}%%%%%%%%%  mV_0 %%%%%%%%%%%%%%%%%%%%%%%%%%%%%%%%%%%%%%%%%%%%%%%%%%%%%%%%
$\mathfrak{m}V_0=\mathbf C_{-k-2}\oplus \mathbf C_{-k+2}$. 
\end{lemma}

\proof
Decompose
$S^2\C^2$ 
into ${\rm U}(1)$-irreducible representation $S^2\C^2|{\rm U}(1)=\C_{2}\oplus\C_0\oplus \C_{-2}$
and using the real structure we have 
$(S^2\C^2)^\R\cong \mathfrak{su}(2),$ $(\C_0)^\R\cong \mathfrak{u}(1)$ therefore
$( \C_{2}\oplus\C_{-2})^\R\cong \mathfrak{m}.$ Then,
\[ \mathfrak m \otimes  V_0 = (\C_{2}\oplus \C_{-2})\otimes \C_{-k} = \C_{-k+2}\oplus\C_{-k-2}\]
The action of $\mathfrak m$ on $V_0$ is then obtained 
by projecting $\mathfrak m \otimes V_0$ back to $S^{|k|+2l}\C^2;$
therefore
\[\mathfrak{m}V_0= \left(\mathfrak{m}\otimes V_0 \right) \cap S^{|k|+2l}\C^2|{\rm U}(1) 
= \mathbf C_{-k-2}\oplus\C_{-k+2}\].
 \qed

\begin{prop}\label{intersection}
Let $W=S^{|k|+2l}\mathbf C^2$ be an irreducible 
 representation %space
of $\text{SU}(2)$ 
with $k,l>0.$ %non-zero integer $k$ and non-negative integer $l$. 
Then 
$\G\mathrm{S}(\mathfrak{m}V_0,V_0)\cap \left(\sigma \mathrm{H}_+(W) \oplus J\sigma \mathrm{H}_+(W)\right)$ 
is the direct sum of 
representations $S^m \mathbf C^2$ of $\text{SU}(2)$ with $m < 4k$ 
appeared in %Proposition \ref{DecH}.
\[
\sigma\mathrm{H}_+(S^{k+2l}\mathbf C^2)\oplus J\sigma 
\mathrm{H}_+(S^{k+2l}\mathbf C^2)
=\bigoplus_{r=0}^{2r\leqq k+2l} S^{2(k+2l)-4r}\mathbf C^2. 
\] 
\end{prop}

Before proving Proposition \ref{intersection} we introduce a necessary technical tool, the {\it contraction operator} $\mathcal C$ and some of its properties.
Let $\{e_1,e_2\}$ be the standard basis of $\mathbf C^2,$ 
 that is,
$\{e_1,e_2\}$ is a unitary basis and satisfies $\omega(e_1,e_2)=1$ 
where $\omega$ is the invariant symplectic form on $\mathbf C^2.$   
 The totally symmetrised product of $n-p$ copies of $e_1$ and $p$ copies of $e_2$
  will be denoted by juxtaposition, $e_1^{n-p}e_2^p.$ 
Defining $u_{n-2p}=\binom{n}{p}^{-\frac{1}{2}}e_1^{n-p}e_2^p,$ % and 
$\{u_n, u_{n-2}, \cdots, u_{-n+2}, u_{-n}\}$ is a unitary, weight basis of 
$S^n\mathbf C^2$.
It is an easy matter to check that $S^{n}\mathbf C^2$ sits in $S^{n-1}\mathbf C^2\otimes \mathbf C^2$ (resp. in $\mathbf C^2\otimes S^{n-1}\mathbf C^2$) as follows:% Next, we have 

\begin{eqnarray*}
e_1^{n-p}e_2^p &=&e_1^{n-p-1}e_2^p\otimes e_1 + e_1^{n-p}e_2^{p-1}\otimes e_2  \in S^{n-1}\mathbf C^2\otimes\mathbf C^2\\
             &=&e_1 \otimes e_1^{n-p-1} e_2^p+e_2\otimes e_1^{n-p} e_2^{p-1} \in \mathbf C^2\otimes S^{n-1}\mathbf C^2
\end{eqnarray*}
Next, we use the symplectic form 
$\omega$ to define an equivariant {\it contraction operator}
$\mathcal C:S^n\mathbf C^2 \otimes S^n\mathbf C^2\to  S^{n-1}\mathbf C^2 \otimes S^{n-1}\mathbf C^2$ given by:
  \[S^n\mathbf C^2\otimes S^n\mathbf C^2 \hookrightarrow S^{n-1}\C^2\otimes
  \overbracket{\mathbf C^2\otimes\mathbf C^2}^{\omega}%}^{\omega}                        %CONTRACTION
  \otimes S^{n-1}\mathbf C^2\to S^{n-1}\mathbf C^2\otimes S^{n-1}\mathbf C^2.\]
Explicitly,
\[
\mathcal{C}\left(e_1^{n-p}e_2^p \otimes e_1^{n-q} e_2^q\right) 
=e_1^{n-p-1} e_2^p\otimes  e_1^{n-q} e_2^{q-1}
-e_1^{n-p} e_2^{p-1}\otimes 
e_1^{n-q-1} e_2^q.
\]
It follows from the equivariance of $\mathcal C,$ and Schur's lemma that
%%%%%%%%%%%%%%%%%%%%%%%%%%%%%%%%%%%%%%%
%\begin{lemma}\label{Cker}
\begin{equation}\label{Cker}\ker\,\mathcal{C}=S^{2n}\mathbf C^2.
\end{equation}
%\end{lemma}
%%%%%%%%%%%%%%%%%%%%%%%%%%%%%%%%%%%%%%%%
Moreover, explicit computation allows to establish the useful formula
\begin{align}\label{cont}
\mathcal{C}^{2r}\Bigl(e_1^{n-p} e_2^p  \otimes e_1^{n-q} e_2^q +e_1^{n-q} e_2^q \otimes &e_1^{n-p}e_2^p \Bigr) \\ 
=\sum_s (-1)^s\binom{2r}{s}   \Bigl(&e_1^{n-p-s} e_2^{p+s-2r}\otimes  e_1^{n-q+s-2r} e_2^{q-s}  \nonumber \\
                              +&e_1^{n-q+s-2r} e_2^{q-s}\otimes e_1^{n-p-1} e_2^{p+s-2r}\Bigr), \nonumber
 \end{align}%
where $s\leqq n-p, q,$ and $s\geqq 2r-p, 2r+q-n.$
If $s$ is not
in this range, 
the corresponding term is regarded as $0$. \\

\noindent{\it Proof of Proposition \ref{intersection}.} 

The real structure map of $\sigma H_+(W)\oplus J\sigma H_+(W),$
interchanges $\pm$ signs on weights. Therefore,  we can assume $k>0$ without loss of generality. Fix a positive
integer $k$ and let $W=S^{k+2l}\mathbf C^2.$
Notice that $\sigma H_+(W)\oplus J\sigma H_+(W)$ is a {\bf complex} $\text{SU}(2)$ representation
and so the following decomposition follows from Clebsch--Gordan formulae:
%%%%%%%%%%%%%%%%%%%%%%%%%
\[
\sigma\mathrm{H}_+(S^{k+2l}\mathbf C^2)\oplus J\sigma \mathrm{H}_+(S^{k+2l}\mathbf C^2)
=\bigoplus_{r=0}^{2r\leqq k+2l} S^{2(k+2l)-4r}\mathbf C^2, 
\]
%%%%%%%%%%%%%%%%%%%%%%%%%%%%%%%%%%%%%%
Consequently, it
can be  identified with %considered as
the space $S^2(S^{k+2l}\mathbf C^2)$ of symmetric powers of  $S^{k+2l}\mathbf C^2$.
Let us make the identification explicit.
%%%%%%%%%%%%%%%%%%%%%%%%%%%%%%%%%%%%%%
Define $X\in \sigma H_+(S^{k+2l}\mathbf C^2)\oplus J\sigma H_+(S^{k+2l}\mathbf C^2)$ as
\[ X=S(u_{-k+2},u_{-k})-S(Ju_{-k+2},Ju_{-k}) 
\]
Since $\sigma(u_{n-2p})=(-1)^pu_{-n+2p}$, we have  
\begin{align*}
2\sigma X = &\langle \cdot, u_{-k+2}\rangle (-1)^{k+l}u_{k}
+\langle \cdot, u_{-k}\rangle (-1)^{k+l-1}u_{k-2} \\
&+\langle \cdot, Ju_{-k+2}\rangle (-1)^{k+l}Ju_{k}
+\langle \cdot, Ju_{-k}\rangle (-1)^{k+l-1}Ju_{k-2} \\
=&(-1)^{k+l}
\left(h(\cdot, u_{-k+2})u_{k}
-h(\cdot, u_{-k})u_{k-2}. \right) \\
=&(-1)^{k+1}\left(h(\cdot, \sigma (u_{k-2}))u_{k}
+h(\cdot, \sigma(u_{k}))u_{k-2} \right).
\end{align*}
Therefore, when we regard 
$\sigma H_+(W)\oplus J\sigma H_+(W)$ as a subspace of 
$S^{k+2l}\mathbf C^2 \otimes S^{k+2l}\mathbf C^2$, 
$X$ corresponds to 
$u_{k-2}\otimes u_k+u_k\otimes u_{k-2}$, that is %and so 

\begin{eqnarray*}
X &\equiv& e_1^{k+2l-(l+1)}e_2^{l+1}\otimes e_1^{k+2l-l}e_2^{l}
+e_1^{k+2l-l}e_2^{l}\otimes e_1^{k+2l-(l+1)}e_2^{l+1}  \\
&=&e_1^{k+l-1} e_2^{l+1}\otimes e_1^{k+l} e_2^{l}
+e_1^{k+l} e_2^{l}\otimes e_1^{k+l-1} e_2^{l+1},
\end{eqnarray*}
up to a constant multiple.  Applying %Lemma
equation (\ref{cont}) to this last expression %if $r\leqq l$, 
\[
\mathcal{C}^{2r}\left( e_1^{k+l-1}e_2^{l+1}\otimes e_1^{k+l} e_2^{l}
+e_1^{k+l} e_2^{l}\otimes e_1^{k+l-1} e_2^{l+1}\right)=0\qquad \text{iff\;} r\geqq l+1
\]
A similar argument is possible for $Y%:$
= \mathrm{S}(u_{-k+2},Ju_{-k})+\mathrm{S}(Ju_{-k+2},u_{-k}).$
It then follows from (\ref{Cker}) that
$$
G\mathrm{S}(\mathfrak{m}V_0,V_0) \bot \bigoplus_{r>l}^{2r\leqq k+2l} S^{2(k+2l)-4r}\mathbf C^2. 
$$
\hfill \qed

\begin{cor}\label{intersection 2}
The orthogonal complement to $\G\mathrm{S}(\mathfrak{m}V_0,V_0) \oplus \mathbf{R}\,Id$
in $\mathrm{S}(W)$ is 
\[ \bigoplus_{r\geqq l+1}^{2r \leqq |k|+2l} S^{2|k|+4l-4r} \mathbf C^2\]
\end{cor}

%%%%%%%%%%%%%%%%%%%%%%%%%%%%%%%%%%%%%%%%%%%%%%%%%%%%%%%%%%%%%%%%%%% XXX new gmod1
\begin{thm}\label{gmod1}
If $f:\C \mathrm P^1\to \mathrm{Gr}_n(\R^{n+2})$ is a full,
($k$,EH($l$)) harmonic map of constant K\"ahler angle, then
  \begin{enumerate}
\item[{$\mathrm{(1)}$}]  $n\leq 2(|k|+2l).$
  \end{enumerate}
  Assuming $n$ to be maximal,
  let $\mathcal M_{k,l}$ be the moduli space up to gauge equivalence of maps,
    and denote  
    its closure by the inner product by $\overline{\mathcal M_{k,l}}.$ Then,
\begin{enumerate}%[resume]
\item[{$\mathrm{(2)}$}] $\mathcal M_{k,l}$  can be regarded as an open bounded convex body 
in \[\bigoplus_{r=l+1}^{2r\leqq |k|+2l}S^{2(|k|+2l-2r)} \mathbf C^2.\]
\item[{$\mathrm{(3)}$}]  The boundary points of $\overline{\mathcal M_{k,l}}$ describe those maps whose images 
  are included in some totally geodesic submanifold
  \[\mathrm{Gr}_p(\mathbf R^{p+2})\subset \mathrm{Gr}_{2(|k|+2l)}(\mathbf R^{2(|k|+2l+1)}),\quad p<2(|k|+2l).\]
\item[{$\mathrm{(4)}$}]  The totally geodesic submanifold $\mathrm{Gr}_p(\mathbf R^{p+2})$ 
can be regarded as the common zero set of some sections of $Q\to \mathrm{Gr}_{2(|k|+21)}(\mathbf R^{2(|k|+2l+1)})$,
which belongs to $(\mathbf{R}^{p+2})^\perp \subset \Gamma( Q)$.
\end{enumerate}
\end{thm}

\begin{proof}
The restriction $n\leq 2(|k|+2l)$ follows from (I) in Theorem \ref{GenDW} and the dimension of the corresponding eigenspace.
\end{proof}

\begin{rem}
  The previous theorem establishes a diffeomorphism between the moduli space $\mathcal{M}_{k,l}$ of full, ($k$,EH($l$)) harmonic mappings
 of constant K\"ahler angle  and the moduli space $\mathcal{M}_k$ of full, holomorphic isometric embeddings of degree $k.$ Due to this,
    proofs of propositions about $\mathcal M_{k,l}$ are, with minor changes, identical to proofs about $\mathcal M_k.$ We now recall some important properties of the moduli space $\mathbf M_{k,l},$ which in virtue of the previous Theorem are derivative from those in \cite{MaNaTa} \S 8. \end{rem}

The complex structure on 
$Q\to \mathrm{Gr}_{2(|k|+2l)}(\mathbf R^{2(|k|+2l+1)})$ induces a similar one on $\mathcal M_{k,l},$
so it is 
a complex 
 submanifold of $\oplus _{r=l+1}^{2r \leqq |k|+2l}S^{2|k|+4l-4r}\mathbf C^2.$  
 As 
 the centraliser of the holonomy group acts on $\mathcal M_{k,l}$ 
with weight $-k,$ we get

\begin{thm}\label{imod}
Let $\mathbf M_{k,l}$ be the moduli space (up to image equivalence) of full, (k,EH(l)) harmonic maps $\mathbf C \mathrm{P}^1\to \mathrm{Gr}_{2(|k|+2l)}(\mathbf R^{2(|k|+2l+1)})$ 
of constant K\"ahler angle. 
Then $\mathbf M_{k,l}=\mathcal M_{k,l}\slash S^1$. 
\end{thm}

Again, as in the holomorphic isometric embedding case,
 $\mathcal M_{k,l}$ 
is a K\"ahler manifold together with an $S^1$--action preserving the K\"ahler structure,
and  is therefore 
equipped with the moment map 
$\mu=|T|^2.$

\begin{cor}\label{imodcor}
  There exists a one--parameter family $\{f_t\},\;t\in[0,1],$ of $\SU(2)$--equivariant image--inequivalent
  isometric minimal immersions of even degree  of  $\C \mathrm P^1$ into complex quadrics where $f_0$ corresponds to
  the standard map.
\end{cor}

\begin{rem}
In this setting $f_1$ would coincide with the {\it real standard map} to be  introduced in the proof of Proposition \ref{pre_rigidity}.
\end{rem}

\section{A one-to-one correspondence between moduli spaces}
From Theorem \ref{gmod1}, the dimension of the moduli space $\mathcal{M}_{k,l}$ of full ($k$,EH($l$)) harmonic maps
of constant K\"ahler angle is independent of the EH constant $l.$
In the present section, we will obtain a one-to-one correspondence between the moduli spaces $\mathcal{M}_{k,l}$ and $\mathcal{M}_{k,l-1}$
for each $k,l.$

For this purpose, we modify the contraction operator $\mathcal C$ (cf, \S 3) as follows:
by Schur's lemma, we can choose an appropriate constant $c_{2n-2r}$ for each irreducible component
$S^{2n-2r} \mathbf C^2\subset \left(S^n \mathbf C^2 \otimes S^n\mathbf C^2 \right) 
,\;1\leq r\leq n$
such that  the modified contraction operator
\begin{equation}\label{modified}\tilde{\mathcal{C}}=\bigoplus_{r=1}^{r=n} c_{2n-2r}\mathcal{C}|_{S^{2n-2r}\mathbf C^2}\end{equation}
preserves the Hermitian inner
product.
The adjoint operator $\tilde{\mathcal C}^{\ast}$ of $\tilde{\mathcal C}$ %is the
coincides with the `inverse' 
$\tilde{\mathcal{C}}^{-1}$
since $\tilde{\mathcal C}^{\ast}\tilde{\mathcal C}$ is a positive Hermitian operator preserving
the Hermitian product.  
 
Next, we introduce the operator norm:
\[
|D|_{op}%:
=max\left\{|\lambda| \, 
:
\, \lambda\,  \text{ is 
  an eigenvalue of }\,D \right\}.
\]
If $f$ is a full ($k$,EH($l$)) harmonic map 
then by Theorem \ref{GenDW}
\[
f([g])=T^{-1}\ker\,ev, 
\]
where 
the positive symmetric automorphism 
$T=(\mathrm{id} +D)^{\frac{1}{2}}\in \text{Aut}\,%(
\mathbf R^{2(|k|+2l+1)}%)
$ 
satisfies 
\begin{equation}\label{DWcond}
(D, \G\mathrm{S}(\mathfrak{m}V_0,V_0))=0.
\end{equation}
Note that $|D|_{op}<1$ due to the positivity of $T$. 

Then $\tilde{\mathcal{C}}^2(D)$ can be regarded as a symmetric endomorphism on 
$\mathbf R^{2(|k|+2(l-1)+1)}$. 
Hence, 
\[
\mathrm{id}+\frac{|D|_{op}}{|\tilde{\mathcal C}^2(D)|_{op}}\tilde{\mathcal C}^2(D)
\]
is also a positive symmetric automorphism on $\mathbf R^{2(|k|+2(l-1)+1)}$. 
Moreover, since $\tilde{\mathcal C}$ is equivariant, $\tilde{\mathcal C}^2(D)$ also satisfies \eqref{DWcond}. 
Consequently, Theorem \ref{GenDW} implies that  
\[
\left(\mathrm{id}+\frac{|D|_{op}}{|\tilde{\mathcal C}^2(D)|_{op}}\tilde{\mathcal C}^2(D)\right)^{-\frac{1}{2}}\ker\,ev, 
\]
is also a full ($k$,EH($l-1$)) harmonic map. %of degree $k,$ but this time with EH constant equal to $l-1$.  
The inverse construction  is
straightforward: we may correspond 
\[
\mathrm{id}+\frac{|D|_{op}}{|\tilde{\mathcal C}^{{\ast}2}(D)|_{op}}\tilde{\mathcal C}^{{\ast}2}(D)
\]
to $\mathrm{id} + D$ to obtain a one-to-one correspondence from $\mathcal M_{k,l}$ to $\mathcal M_{k,l+1}$. Thus we have established the following

\begin{thm}\label{flight.theorem} There is a one-to-one correspondence %(the \emph{flight of moduli})
  between $\mathcal M_{k,l}$ and $\mathcal M_{k,l-1}$ which associates the gauge--equivalence class of full ($k$,EH($l$)) harmonic maps
  determined by $T=(\mathrm{id} + D)^{\frac12}$ to the gauge--equivalence class of full ($k$,EH($l-1$)) harmonic maps
  determined by
  \[\left(\mathrm{id}+\frac{|D|_{op}}{|\tilde{\mathcal C}^2(D)|_{op}}\tilde{\mathcal C}^2(D)\right)^{\frac12}\]
where $\tilde{\mathcal C}$ denotes the modified contraction operator (\ref{modified}).
\end{thm}

\section{Rigidity of the real standard map}                                               
%XXX We should need a sentence saying WHAT are we going to do next XXX

An $\SU(2)$--irreducible representation is a \emph{class--one representation of} the pair $(\SU(2),\mathrm{U}(1))$,
if it contains non--zero $\mathrm{U}(1)$--invariant elements.

\begin{prop}\label{pre_rigidity} $\qquad$
\begin{enumerate}
\item  Let $S^{|k|+2l}\mathbf C^2$ be the $l$-th eigenspace of the vector bundle 
$\mathcal O(k)\to \mathbf C\mathrm P^1$ 
and $V_0$ the $\mathrm{U}(1)$--representation regarded 
as its
standard fibre. 
Then, $\G\mathrm{H}(V_0,V_0)=\mathrm{H}(S^{|k|+2l}\mathbf C^2).$
\item
Let $S^{2|k|+2l}\C^2$ be the indicated representation space of $\mathrm{SU}(2)$, 
of which an invariant real subspace is $S^{|k|+l}_0\R^3\cong \R^{2(|k|+l)+1}$ 
and $V_0$ the $\mathrm{U}(1)$--representation regarded 
as the standard fibre for  $\mathcal{O}(2k)\to \C \mathrm P^1.$ 
Then, 
$\G\mathrm{S}(V_0,V_0)=\mathrm{S}(S^{|k|+l}_0\R^3)$. 
\end{enumerate}
\end{prop}

\begin{proof}
  \noindent
  $(1)$ The proof is by {\it reductio ad absurdum} and follows, with minor modifications, the same lines as Theorem 5.4 in \cite{MaNaTa}. We sketch the argument. The 
  $\SU(2)$-module $W=S^{|k|+2l}\mathbf C^2$ decomposes under $\mathrm{U}(1)$ as 
\[S^{|k|+2l}\mathbf C^2=\C_{-|k|-2l}\oplus\C_{-|k|-2l+2}\oplus\dots\oplus\C_{|k|+2l}.\] 
Then, $V_0=\C_{-k}$ by Frobenius reciprocity and the invariance of the Hermitian inner product,
and the decomposition is {\it normal} (cf \cite{DoC-Wal}, or \cite{MaNaTa} \S 6).
Consider a class--one representation of $(\SU(2),\mathrm{U}(1))$ in the orthogonal complement to $\mathrm{GH}(V_0,V_0)$ in $\mathrm{H}(W),$
and let $C$ be a non-zero $\mathrm{U}(1)$-invariant element in it.
Polarisation of the orthogonality condition $(C,gH(v_1,v_2))_{\mathrm{H}(W)}=0$ %
leads to $(Cgv_1,gv_2)=0$ for all $g\in \mathrm{SU}(2),\, v_1,v_2\in V_0.$
A positive Hermitian operator $T$ is then defined by $T^2=\mathrm{id}+C,$ for $C$ small enough. Being
$\mathrm{U}(1)$-equivariant, Schur's lemma implies that $T=\mathrm{id}$ and so $C=0$ against the hypothesis.
Therefore, 
every class--one subrepresentation of $(\SU(2),\mathrm{U}(1))$ %$(G,K)$
in $\mathrm{H}(W)$ is %included
in $\mathrm{GH}(V_0,V_0).$

$(2)$ 
Regarding $\mathbf{C}P^1$ as $\SU(2)/\mathrm{U}(1),$ the space of sections $\Gamma(\mathcal O(2k))$ becomes an $\SU(2)$--module
and the $l$-th eigenspace %can be regarded as
is identified with $S^{2|k|+2l}\C^2.$ 
This decomposes under $\mathrm{U}(1)$ as
%%%%%%%%%%%%%%%%%%%%%%%%%%%%%%%%%%%%%%%%%%%%%%%%%%%%%%
\begin{equation}\label{su2_to_u1} S^{2|k|+2l}\C^2 = \bigoplus_{r=0}^{2|k|+2l} \C_{2|k|+2l-2r} 
\end{equation}
%%%%%%%%%%%%%%%%%%%%%%%%%%%%%%%%%%%%%%%%%%%%%%%%%%%%%%%%%%%%%%%%%%%%%%%%%%%%%%%%
By the same arguments above, the typical fibre of $\mathcal O(2k)\to \mathbf C\mathrm P^1$ is identified with % as a subspace
$\mathbf C_{-2|k|}$ in (\ref{su2_to_u1}). %the previous decomposition. %by Frobenius reciprocity and the invariant Hermitian inner product. 
Although $S^{2|k|+2l}\C^2$ has a real invariant subspace $S^{|k|+l}_0\R^3$ of dimension $2|k|+2l+1,$
the irreducible components in the right-hand side of (\ref{su2_to_u1}) are not invariant under the real structure $\sigma,$
but $\sigma(\C_{2|k|+2l-2r})=\C_{-2|k|-2l+2r}.$
Therefore $\sigma$ leaves $(\C_{2|k|+2l-2r}\oplus \C_{-2|k|-2l+2r})$ stable for each $r=0,\dots,|k|+l+1,$
which splits in two isomorphic real irreducible $\text{U}(1)$--modules,
denoted $(\C_{2|k|+2l-2r}\oplus \C_{-2|k|-2l+2r})^\R.$ 
If $r=|k|+l$, then $(\mathbf C)^{\mathbf R}=\mathbf R$,  the trivial real representation. 
Thus 
%%%%%%%%%%%%%%%%%%%%%%%%%%%%%%%%%%%%%%%%%%%%%%%%%%%%%%%%%%%%%%%%%%%%%%%%%%%%%%%%%%%%%%%%%%%%%%%%%%%%%%%
\begin{equation}\label{real}
  S_0^{|k|+l}\R^3%|{\rm U}(1) 
= \bigoplus_{r=0, r\not=|k|+l}^{2|k|+2l} (\C_{2|k|+2l-2r}\oplus\C_{-2|k|-2l+2r})^\R\oplus \mathbf R.
\end{equation}
%%%%%%%%%%%%%%%%%%%%%%%%%%%%%%%%%%%%%%%%%%%%%%%%%%%%%%%%%%%%%%%%%%%%%%%%%%%%%%%%%%%%%%%%%%%%%%%%%%%%%%
The space $S^{|k|+l}\R^3$ globally generates $\mathcal O(2k)\to \mathbf C\mathrm P^1,$ and so determines 
a \emph{real standard map} $f_0:\C \mathrm P^1 \to \mathrm{Gr}_{2k-1}(\R^{2k+1})$  which turns out to be an
($2k$,EH($l$)) isometric minimal immersion
by Lemma 2.3 in \cite{MaNaTa}. 
It is then possible to define the
adjoint of the evaluation map,
\[ev^*_{[e]}:\mathcal{O}(2k)\to \underline{S^{|k|+l}\R^3},\]
such that %which at the identity of $\C P^1$ determines a mapping 
its image at the reference point of $\C \mathrm P^1,$
is $(\C_{2|k|}\oplus\C_{-2|k|})^\R.$ 

Now, \eqref{real} gives the normal decomposition of $S^{|k|+l}_0\R^3$ 
where  $V_0=(\C_{-2k}\oplus\C_{2k})^\R.$ 
The space of symmetric endomorphisms of $S^{|k|+l}_0\R^3$ can be identified
using representation theory as in \cite{MaNaTa} \S 4 to give
\[
\mathrm{S}(S^{|k|+l}_0\R^3)\subset \mathrm{End}(S^{|k|+l}_0\R^3)
=S^{|k|+l}_0\R^3\otimes_\R (S^{|k|+l}_0\R^3)^*\cong \otimes^2 S^{|k|+l}_0\R^3,
\]
\[
\mathrm{S}(S^{|k|+l}_0\R^3)=\bigoplus_{r=0}^{|k|+l} S^{4|k|+4l-4r}_0\R^3 \subset \otimes^2 S^{|k|+l}_0\R^3 
= \bigoplus_{r=0}^{2|k|+2l} S^{4|k|+4l-2r}_0\R^3.\]

\end{proof}

The real standard map induced by 
$\mathcal{O}(2k)\to \C \mathrm P^1$
and 
$S^{|k|+l}_0\R^3$ has just been depicted in the above proof. %of Theorem \ref{pre_rigidity} (b) above.
Since its
deformation space is,  %of the standard map
up to gauge equivalence, %in correspondence with
$\mathrm{GS}(V_0,V_0)^\perp\subset S(S^{|k|+l}_0\R^3) $ we
obtain the following

%%%%%%%%%%%%%%%%%%%%%%%%%%%%% RIGIDITY %%%%%%%%%%%%%%%%%%%%%%%%%%%%%%%%%%%%%%%%%%
\begin{cor}\label{rigidity}
Let $S^{|k|+l}_0\R^3\cong\R^{2(|k|+l)+1}$ be the real invariant subspace of the $\SU(2)$-module
  $S^{2|k|+2l}\C^2.$
If $f:\mathbf C\mathrm P^1 \to \mathrm{Gr}_{2(|k|+l)-1}(\mathbf R^{2(|k|+l)+1})$ is a ($2k$,EH($l$)) isometric minimal 
immersion, %of degree $2k,$ and  EH constant $l$,  
then it is the standard map induced by $S^{|k|+l}_0\R^3$ up to gauge equivalence.
\end{cor} 

%%%%%%%%%%%%%%%%%%%%%%%%%%%%%%%%%%%%%%%%%%%%%%%%%%%%%%%%%%%%%%%%%%%%%%%%%%%%%%%%%
\section{Rigidity of totally real EH harmonic maps}
%%%%%%%%%%%%%%%%%%%%%%%%%%%%%%%%%%%%%%%%%%%%%%%%%%%%%%%%%%%%%%%%%%%%%%%%%%%%%%%%
We start our discussion by recalling the relation between certain special class
of harmonic maps, and totally real immersions: Let $f:\mathbf C\mathrm P^1 \to \mathrm{Gr}_n(\mathbf R^{n+2})$
be a harmonic map of degree $0$ with constant energy density 
satisfying the gauge condition (\ref{gauge.condition}). 
Since the canonical connection on the trivial complex line bundle is flat, 
we see from the gauge condition that
\[
f^{\ast}\omega_n=0. 
\]
Since $f$ is conformal, we conclude that it is a totally real isometric immersion up to a constant multiple. 
Conversely, if $f:\mathbf C\mathrm P^1 \to \mathrm{Gr}_n(\mathbf R^{n+2})$ 
is an isometric minimal totally real immersion, 
then $f$ satisfies the gauge condition above.
Thus $f$ is a harmonic map of degree $0,$ 
with constant energy density, 
satisfying (\ref{gauge.condition}). 
 
We will show the rigidity of totally real minimal immersion with the 
EH condition in this section. 

\begin{thm}
  Let $f:\mathbf C\mathrm P^1\to \mathrm{Gr}_n(\mathbf R^{n+2})$ be a totally real, full,  ($0$,EH($l$)) 
  minimal immersion. 
    Then $n=4l$ and $f$ is image-equivalent to the standard map.
\end{thm} 

\begin{proof}
  Since $f$ is   EH harmonic,   it follows from  the 
  generalisation of the Theorem of Takahashi \cite{Na-13} that 
$\mathbf R^{n+2}$ is an eigenspace of the Laplacian acting on 
$\mathbf R^2$-valued functions.  
Thus, $\mathbf R^{n+2}$ is a subspace of $S^l\mathbf C^2$, 
where $l$ is an even number. 
To define the standard map $f_0$, we consider the weight decomposition of 
$S^l\mathbf C^2$:
\[
S^l\mathbf C^2=\mathbf C_l \oplus \mathbf C_{l-2} \oplus \cdots \oplus \mathbf C_0 \oplus \cdots 
\oplus \mathbf C_{-l}.  
\]
Then $f_0([g])=g\mathbf C_0^{\bot}$.  
The normal decomposition of the standard map is given  by 
\[
\text{\rm Im}\,B_p=\mathbf C_{2p} \oplus \mathbf C_{-2p}.
\]
Then the same argument as in the proof of Theorem 5.4 
in \cite{MaNaTa} gives the result. 
\end{proof}
Actually, the result would still hold at the level of gauge-equivalence
of maps.

The EH condition is indispensable for the previous rigidity argument,
as dropping it leads to the following counterexample  by
Wang and Jiao \cite{JWanJia}.

We use an invariant real subspace $\mathbf R^{2m+1}\subset S^{2m}\mathbf C^2$
to define a totally real isometric minimal %totally real
immersion 
$f_1:\mathbf C\mathrm P^1 \to \mathrm{Gr}_{2m}(\mathbf R^{2m+2})$. 
The weight decomposition of $\mathbf R^{2m+1}$ is %as follows:
\[
\mathbf R^{2m+1}=\mathbf C_{2m} \oplus \mathbf C_{2m-2} \oplus \cdots \mathbf C_2 \oplus 
\mathbf R_0.
\]
Consider the orthogonal direct sum of $\mathbf R^{2m+1}$ and a trivial representation 
$\mathbf R$. 
The orthogonal complement of $\mathbf R_0 \oplus \mathbf R$ in $\mathbf R^{2m+1} \oplus \mathbf R$ is denoted by 
$(\mathbf R_0 \oplus \mathbf R)^{\bot}$. 
Then we define $f_1$ as 
\[
f_1([g]):=g(\mathbf R_0 \oplus \mathbf R)^{\bot}=g\mathbf R_0^{\bot}, 
\]
where $\mathbf R_0^{\bot}$ denotes the orthogonal complement of 
$\mathbf R_0$ in $\mathbf R^{2m+1}.$ 
The orientation of $\mathbf R_0 \oplus \mathbf R$ is determined by this ordering. 

Now, 
$f_1$ is of degree $0$ and the induced connection on the pull-back 
of the universal quotient bundle is a product connection. 
By definition of $f_1$, the mean curvature operator $A$ has $0$ as eigenvalue. 
Moreover,  $A$ is parallel 
because its eigenvalues 
are constant, and the eigenspace decomposition is invariant under the connection.  
The generalisation of Theorem of Takahashi yields that $f_1$ is %a
harmonic. Then, by the above argument, $f_1$ 
is a totally real, isometric minimal 
immersion  with $\text{det}\,A=0$ and $\nabla A=0$. 

Notice that $f_0$ and $f_1$ are not
  image-equivalent 
  because their mean curvature
  operators are different. Indeed, the image-equivalence class of $f_1$ is
characterised by the following
\begin{thm}
Let $f:\mathbf C\mathrm P^1\to \mathrm{Gr}_n(\mathbf R^{n+2})$ be 
a totally real, full, minimal immersion with 
$\text{\rm det}\,A=0$ and $\nabla A=0$. 
Then, $n$ is an even integer and 
$f$ is image equivalent to $f_1$.
\end{thm}

\begin{proof}
The pull-back of the universal quotient bundle has the eigenspace decomposition of $A$, 
which is invariant under the pull-back connection, because $A$ is parallel. 
Since $\text{\rm det}\,A=0$, its only eigenvalues are 
$0$ and $-2e(f).$  
The orthogonal decomposition of $V \to \mathbf C\mathrm P^1$ is denoted by 
$V_1\oplus V_2$  corresponding to eigenvalues $0$ and $-2e(f)$, respectively.  
The real vector space $\mathbf R^{n+2}$ determines a family of sections 
of $V_1\to \mathbf C\mathrm P^1$ which, 
by the generalisation of the Theorem of Takahashi, are constant. 
Hence the image of $V_1 \to \mathbf C\mathrm P^1$ under the adjoint of the evaluation 
homomorphism gives a subspace $U_1$ of $\mathbf R^{n+2}$. 
It follows from the fullness of $f$ that  $U_1$ is of dimension one 
and so, $V_1\to \mathbf C\mathrm P^1$ is identified with $\underline{U_1}\to \mathbf C\mathrm P^1$. 
 
The orthogonal complement of $U_1$ is denoted by $U_2$. 
Using again the generalisation of the Theorem of Takahashi,
there exists an integer $m$ 
such that $e(f)=2\pi\{2m(m+1)\}$. 
Moreover, it follows from the fullness of $f$ that $U_2$ is a subspace of 
$\mathbf R^{2m+1}$ which is an invariant real subspace of 
$S^{2m}\mathbf C^2$.  
Notice that we now consider {\it real}-valued functions and so, 
the eigenspace of the Laplacian is $\mathbf R^{2m+1}$ for $m\in \mathbf Z_{\geqq 0}$. 

Therefore, by the discussion at the beginning of the section, the problem is equivalent to that of
classifying an ($0$,EH($m$)) harmonic map $\mathbf C\mathrm P^1\to\mathbf R\mathrm P^{2m},$ %of degree $0,$ EH constant $m,$
fulfilling the gauge condition.

The standard map induced by a rank-one, trivial vector bundle 
and  $\mathbf R^{2m+1}$ has the normal decomposition:
\[
\text{\rm Im}\,B_p=\mathbf C_{2p}, 
\qquad p\geq 1
\]
and then again the result follows as in the proof of  Theorem 5.4 \cite{MaNaTa}.
\end{proof}

\section{Applications}

In this final section we use our classification theorem 
(in particular, Corollary \ref{imodcor}) 
to give a new proof of the following classical result, originally
  stated in its different incarnations by Calabi \cite{Cal}, and do Carmo \& Wallach \cite{DoC-Wal}. 
\begin{thm}
  \label{CDW}
Let $i:S^2 \to S^n \subset \mathbf R^{n+1}$ 
be a full isometric minimal immersion, 
where $S^n$ is the unit sphere. 
Then $i$ is $\text{SU}(2)$-equivariant. 
\end{thm}
The proof of the theorem will follow from a reinterpretation (in the sense of \cite{Na-13})
of the well-known theory developed by Ruh \& Vilms \cite{Ruh-Vil}.
To this end, let $M$ be a Riemannian manifold and  
$I:M\to \mathbf R^{N}$ an isometric immersion. 
Using the inner product on $\mathbf R^{N}$, 
we consider a bundle homomorphism  
$ev:\underline{\mathbf R^{N}}  \to TM.$
In fact, $ev$ is the adjoint homomorphism of 
$dI:TM \to \underline{\mathbf R^N}$. 
Then, we have an exact sequence of vector bundles: 
\[
0 \to NM \to \underline{\mathbf{R}^{N}} \to TM \to 0,
\]
where  
$NM \to M$ is the {\it normal bundle} of $M.$ % to $I:M \to \mathbf R^{N}$. 
Denote the second fundamental form of the tangent %vector
bundle by $K\in \Omega^1(TM^{\ast}\otimes NM).$ 
Since $I$ is an isometric immersion,  $K$ is also regarded as the second fundamental form of 
submanifold geometry. 
Hence, it satisfies %we have 
\[
K_X Y =K_Y X, \quad X,Y \in TM. 
\]
The bundle homomorphism $ev,$ together with %and
the orientation of $M$ induces
a Gauss map $f:M \to \mathrm{Gr}_{p}(\mathbf R^{N})$:
\[
f(x)=\ker \,ev_x, \quad x \in M,  
\]
where $p=N -\dim \,M$.  
In this context (cf \S 2) the pull-back of the universal quotient bundle
$Q\to \mathrm{Gr}_p(\mathbf R^N)$ coincides with 
the tangent bundle of $M,$  
and 
$K$ can also be 
regarded as the pull-back of the second fundamental form 
of $Q\to \mathrm{Gr}_{p}(\mathbf R^{N})$ in the exact sequence
\[ 0 \to S \to \underline{\mathbf{R}^{N}} \to Q \to 0. \]
Let $\mathbf n$ denote the {\it mean curvature} of $I:M \to \mathbf R^{N}.$
 Explicitly,
$\mathbf n= \sum K_{e_i} e_i$, where $\{e_1,e_2, \cdots \}$ is an orthonormal frame of $%TM \to
M$. 
The Gauss-Codazzi equations, \cite{Kob}, p.23,  %for vector bundles
then yield %that 
\[
\nabla_X \mathbf n= \sum \nabla_X \left(K_{e_i} e_i\right)=
\sum (\nabla_X K)(e_i,e_i)=\sum (\nabla_{e_i}K)(e_i,X)
=K_{\tau(f)}X
\]
where $\tau(f)$ is the {\it tension field} of the Gauss map $f$. 
Hence, we recover Ruh \& Vilms result \cite{Ruh-Vil} %have a Theorem of Ruh-Vilms. 
\begin{thm}
Let $M$ be a Riemannian manifold and  
$I:M\to \mathbf R^{N}$ an isometric immersion. 
The Gauss map is denoted by $f:M \to Gr_p(\mathbf R^N)$.
Then 
the mean curvature of $I$ is parallel if and only if 
$f$ is a harmonic map.
\end{thm}
Moreover, since $I$ is an isometric immersion, the Gauss map $f$ 
satisfies the gauge condition; in other words,
the pull-back connection is the Levi-Civita connection on $V \cong TM \to M$.  

Finally, we use the Gauss--Codazzi equations
to compute the mean curvature operator $A$ of $f:M \to \mathrm{Gr}_p(\mathbf R^N)$:
\[
AX=-\sum K_{e_i}^{\ast}K_{e_i}X=K_X^{\ast}\mathbf n-Ric^M(X), 
\] 
where $Ric^M$ is the Ricci operator of $M$. 

Next, let $i:M\to S^n \subset \mathbf R^{n+1}$ 
be an isometric minimal immersion, 
where $S^n$ is the unit sphere. 
By composition, we get an isometric immersion 
$I:M \to \mathbf R^{n+1}$ with parallel mean curvature $\mathbf n$. 
In this case, $\mathbf n(x)=-mI(x)$, where $m=\dim \,M$,  
and so, $K_X^{\ast}\mathbf n=-mX$.  
Summarising,

\begin{lemma}
Let $i:M\to S^n \subset \mathbf R^{n+1}$ 
be an isometric minimal immersion. 
Using the standard embedding $S^n \subset \mathbf R^{n+1}$, 
we get an isometric immersion 
$I:M \to \mathbf R^{n+1}$ with parallel mean curvature. 
 The Gauss map of $I$ is an EH harmonic mapping if and only if $M$ is an Einstein manifold. 
\end{lemma}
Now, we can proceed with the\\

\noindent {\it Proof of Theorem \ref{CDW}.}\\
By composition, we consider an isometric immersion 
$I:S^2 \to \mathbf R^{n+1}$ with parallel mean curvature. 
Since $S^2$ is an Einstein manifold, 
the Gauss map of $I$ denoted by $f$ is an EH harmonic map 
with the gauge condition. 
Since the pull-back of the universal quotient bundle is identified with the tangent bundle,  
$f$ is of degree $2$.  
We use Theorem \ref{imod} and Corollary \ref{imodcor} to conclude that 
$f$ is an $\text{SU}(2)$-equivariant map. 
Since $f$ can be considered as the differential of $I$, 
$I$ and $i$ themselves are $\text{SU}(2)$-equivariant maps.
\hfill \qed


\begin{thebibliography}{34}

%\bibitem{Ban-Ohn} 
%S.Bando and Y.Ohnita, 
%{\it Minimal 2-spheres with constant curvature in $\mathbf P_n(\mathbf C)$}, 
%J. Math. Soc. Japan {\bf 39} (1987), 477--487

\bibitem{Bolton.etal}
  J. Bolton, G.R. Jensen, M. Rigoli, L.M. Woodward,
  {\it On conformal minimal immersions of $S^2$ into $\mathbf C \mathrm P^n$},
  Mathematische Annalen, 279.4 (1987/88): 599-620.
  
\bibitem{Cal}
E. Calabi, 
{\it Isometric Imbedding of Complex Manifolds}, 
Ann. of Math. {\bf 58} (1953), 1--23

%\bibitem{ChiZhe}
%Q.S.Chi and Y.Zheng, 
%{\it Rigidity of pseudo-holomorphic curves of constant curvature in Grassmann manifolds}, 
%Trans. Amer.Math. Soc. {\bf 313} (1989), 393-406

\bibitem{Chen.Ogiue}
B.Y. Chen, K. Ogiue,
\emph{On  totally  real  submanifolds},
Trans. Amer. Math. Soc. {\bf 193} (1974) 257-266.

\bibitem{Chern.Wolfson}
  S.S. Chern, J. Wolfson,
  \emph{Minimal surfaces by moving frames},
  Am. J. Math. {\bf 105}, 59-83 (1983).

\bibitem{DoC-Wal}
M.P. do Carmo and N.R. Wallach, 
{\it Minimal immersions of spheres into spheres},  
Ann.Math. {\bf 93} (1971), 43--62.

\bibitem{E-L}
J. Eells and L. Lemaire, 
{\it A report on Harmonic maps}, 
Bull.London.Math.Soc. {\bf 10} (1978), 1--68.

%\bibitem{Ee-Sam}
%J.Eells and J.H.Sampson, 
%{\it Harmonic mappings of Riemannian manifolds}, 
%Amer. J. Math. {\bf 86} (1964), 109-160

%\bibitem{FJXX}
%J.Fei, X.Jiao, L.Xiao and X.Xu, 
%{\it On the Classification of Homogeneous 2-Spheres in Comple Grassmannians}, 
%Osaka J. Math {\bf 50} (2013), 135-152


%\bibitem{LiYu}
%Z.Q.Li and Z.H.Yu, 
%{\it Constant curved minimal 2-spheres in G(2,4)}, 
%Manuscripta Math. {\bf 100} (1999), 305-316

\bibitem{Kob}
S. Kobayashi, 
``Differential Geometry of Complex Vector Bundles", 
Iwanami Shoten and Princeton University, Tokyo (1987).

\bibitem{Kog-Nag}
  I. Koga, Y. Nagatomo,
 {\it A Study of Submanifolds of the Complex Grassmannian Manifold with Parallel Second Fundamental Form},
  Tokyo J. of Math. {\bf 39}, 1 (2016), 173-185.



%\bibitem{kodaira} K. Kodaira, {\it On K\"ahler varieties of restricte dtype (an intrinsic characterisation of algebraic varieties)},
%Ann. of Math. (2) {\bf 60} (1954), 28-48.

\bibitem{MaNaTa} % OMsays: I added the reference to our previous paper
  O. Macia, Y. Nagatomo, M. Takahashi,
  {\it Holomorphic isometric embeddings of the projective line into quadrics},
  to appear in Tohoku Math. J.


\bibitem{Na-13}
Y. Nagatomo, 
{\it Harmonic maps into Grassmannian manifolds},
arXiv: mathDG/1408. 1504.


\bibitem{Ruh-Vil}
E.A. Ruh and J. Vilms, 
{\it The tension field of the Gauss map}, 
Transactions of the American Mathematical Society, {\bf 149}, (1970), 569-573.

%\bibitem{TTaka}
%T.Takahashi, 
%{\it Minimal immersions of Riemannian manifolds}, 
%J. Math. Soc. Japan {\bf 18} (1966), 380-385


\bibitem{takeuchi}
M. Takeuchi,
  {\it Modern spherical functions.}% Translated from the 1975 Japanese original by Toshinobu Nagura.
  Translations of Mathematical Monographs, 135. American Mathematical Society, Providence, RI, 1994. %x+265 pp. ISBN: 0-8218-4580-2
  
%\bibitem{Tot}
%G.Toth, 
%{\it Moduli Spaces of Polynomial Minimal Immersions between Complex Projective Spaces}, 
%Michigan Math.J. {\bf 37} (1990), 385--396

\bibitem{JWanJia} 
J. Wang and X Jiao,  
{\it Totally real minimal surfaces in the complex hyperquadrics}, 
Diff.Geom. and its Appl.  {\bf 31} (2013), 540--555.

%\bibitem{Wolf}
%J,G.Wolfson 
%{\it Harmonic maps of the two-sphere into the complex hyperquadric}, 
%J. Diff. Geo. {\bf 24} (1986), 141-152

\end{thebibliography}
\end{document}